\documentclass[12pt]{article}
\usepackage[utf8]{inputenc}
\usepackage{amsmath}
\usepackage{amssymb}
\usepackage{mathrsfs}
\usepackage{amsthm}
\usepackage{cite}
\usepackage{hyperref}
\usepackage{enumerate}
\usepackage{combelow}
\usepackage[left=0.7in,right=0.7in,top=0.7in,bottom=0.7in]{geometry}
\usepackage{titlesec}
\titleformat*{\section}{\large\bfseries}
\newtheorem{theorem}{Theorem}[section]
\newtheorem{lemma}[theorem]{Lemma}
\newtheorem{corollary}[theorem]{Corollary}
\newtheorem{definition}[theorem]{Definition}
\newtheorem{example}[theorem]{Example}

\numberwithin{equation}{section}
\title{Coupled best proximity point theorems for $p$-cyclic $\phi$-contraction and $p$-cyclic Kannan nonexpansive mappings}
\author{\large  Parveen Kumar$^1$ and Ankit Kumar$^2$\footnote{Corresponding author} \\
{\small $^1$,$^2$Department of Mathematics, Punjab Engineering College (Deemed to be University)  }\\
{\small Chandigarh-160012, India}\\{\small E-mail: $^1$parvveenbeniwal11@gmail.com, $^2$ankitkumar@pec.edu.in}\\ 
 }
\date{}
\begin{document}

\maketitle
\begin{abstract}
In this paper, the notions of $p$-cyclic $\phi$-contraction and $p$-cyclic Kannan nonexpansive mappings are introduced, and the existence of coupled best proximity points for such mappings is established.\\
\textbf{Mathematics Subject Classification.} 47H10, 47H09, 41A65.\\
\textbf{Keywords.} Coupled best proximity point, $p$-cyclic $\phi$-contraction mapping, $p$-cyclic Kannan nonexpansive mapping. 
\end{abstract}

\section{Introduction and preliminaries}
Throughout this paper, the set of natural numbers is denoted by $\mathbb{N}$. Let $A$ and $B$ be two nonempty subsets of Banach space $X$. We define
\begin{align*}
dist(A,B) & = \inf\{\Vert a-b \Vert: a \in A \thinspace \thinspace \mbox{ and } \thinspace b \in B \},\\
A_0 &=  \{ a \in A :\Vert a-b \Vert =dist(A,B) \thinspace \thinspace \mbox{for some} \thinspace b \in B \},\\ 
B_0 & =  \{ b \in B :\Vert a-b \Vert =dist(A,B) \thinspace \thinspace \mbox{for some} \thinspace a \in A \}.
  \end{align*}
  
Let $S:(A \times B) \cup (B \times A) \rightarrow A \cup B$ be a mapping. Then $(a,b) \in A \times B$ is said to be a coupled best proximity point of $S$ if it satisfies $\Vert a-T(a,b)\Vert=dist(A,B)$ and $\Vert b-T(b,a) \Vert=dist(A,B)$. Interestingly, under the assumption that $A \cap B = \emptyset$ this extends the concept of a coupled fixed point \cite{7}. When $ A \cap B \neq \emptyset $, the notion of a coupled best proximity point reduces to a coupled fixed point.

Best approximation results ensure the existence of approximate solutions, though they may not always be optimal. In contrast, best proximity point results yield optimal approximate solutions. Several authors (see \cite{2,3,6,11,12,13,1} have studied best proximity points for various contraction and nonexpansive mappings in the setting of metric spaces and  Banach spaces. The following result by Kirk et al. \cite{10} ensures the nonemptiness of $A_0$ and $B_0$.
\begin{lemma}
\cite[Lemma 3.2]{10} Let $A$ and $B$ be nonempty, closed and convex subsets of a reflexive Banach space $X$. Suppose that $A$ is bounded. Then $A_0$ and $B_0$ are nonempty.
\end{lemma}
A Banach space $X$ is said to be uniformly convex if there exists a strictly increasing function $\delta:(0,2] \rightarrow [0,1]$ such that for all $x,y,x \in X$, $R>0$ and $r \in [0,2R]$ satisfying $\Vert x-z \Vert \leq R$, $\Vert y-z \Vert \leq R$ and $\Vert x-y \Vert \geq r$, we have $\Big\Vert \frac{x+y}{2}-z \Big\Vert \leq \Big(1-\delta(\frac{r}{R} \Big) \Big)R$. Let $(X, \Vert . \Vert)$ a Banach space. Define a norm on $X \times X$ by $\Vert (x,y) \Vert = \max \{\Vert x \Vert,\Vert y \Vert\}$.
 
Al-Thagafi and  Shahzad \cite{15} introduced the notion of cyclic  $\phi$-contraction mappings and studied the existence of a best proximity point of such mappings.
\begin{definition}
\emph{\cite[Definition 1]{15} Let $A$ and $B$ be nonempty subsets of a metric space $(X,d)$ and let $\phi: [0,\infty) \rightarrow [0, \infty)$ is strictly increasing mapping. A mapping $T:A \cup B \rightarrow A \cup B$ is called cyclic $\phi$-contraction  if the following conditions are satisfied:}

\emph{(i) $T(A) \subset B$ and $T(B) \subset A$,}

\emph{(ii) $d(Tx,Ty) \leq d(x,y)-\phi(d(x,y))+\phi(dist(A,B))$ for all $x \in A$ and $y \in B$.}
\end{definition}
\begin{theorem}
\cite[Theorem 8]{15} Let $A$ and $B$ be nonempty subsets of a uniformly convex Banach space $X$ such that $A$ is closed and convex and let $T:A \cup B \rightarrow A \cup B$ be a cyclic $\phi$-contraction mapping. For $x_0 \in A$, define $x_{n} = Tx_{n-1}$ for each $n \geq \mathbb{N}$. Then there exists unique $x \in A$ such that $x_{{2n}} \rightarrow x$, $T^2x =x$ and $\Vert x-Tx \Vert = dist(A,B)$.
\end{theorem}
Gabeleh\cite{17} introduced the notion of cyclic Kannan nonexpansive mapping and studied best proximity point of such mappings without any geometric property.
\begin{definition}
\emph{\cite[Definition 2.2]{17} Let $(A,B)$ be a nonempty pair in a normed linear space $X$. A mapping $T :A \cup B \rightarrow A \cup B $ is said to be cyclic Kannan nonexpansive mapping if it satisfies the following conditions:}

\emph{(i) $T(A) \subseteq B$ and $T(B) \subseteq A$,}
  
\emph{(ii) $\Vert Tx-Ty \Vert \leq \frac{1}{2} \big [ \Vert x - Tx \Vert + \Vert y- Ty \Vert ]$  for all $(x,y) \in A \times B$.}
\end{definition}
\begin{theorem}
\cite[Theorem 3.1]{17} Let $(A,B)$ be a nonempty, weakly compact and convex pair of subsets of a normed linear space $X$. Suppose $T:A \cup B \rightarrow A \cup B$ is a cyclic Kannan nonexpansive mapping such that  $\Vert T^2x - Tx \Vert < \Vert x-Tx \Vert$ for all $x \in A \cup B$ with $dist(A,B)<\Vert x-Tx \Vert$. Then $T$ has a best proximity point.
 \end{theorem}
Gupta and Rohilla \cite{8} introduced the notion of $p$-cyclic contraction mappings and established the existence of coupled best proximity point of such mappings.
\begin{definition}
\emph{\cite[Definition 2.1]{8} Let $A$ and $B$ be nonempty subsets of a Banach space $X$. A mapping $T :(A \times B) \cup (B \times A) \rightarrow A \cup B$ is called a $p$-cyclic contraction mapping if it satisfies the following conditions:}

 \emph{(i) $T(A,B) \subseteq B$ and $T(B,A) \subseteq A$},
 
\emph{(ii) $ \Vert T(x_{1},y_{1})-T(x_{2},y_{2}) \Vert \leq
\lambda  \Vert (x_{1},y_{1}) - (x_{2},y_{2}) \Vert + (1- \lambda )dist(A,B)$ for some $\lambda \in (0,1)$.}
\end{definition}
\begin{theorem}
\cite[Theorem 2.6]{8} Let $A$ and $B$ be nonempty closed and convex subsets of a uniformly convex Banach space $X$. Let $T :(A \times B) \cup (B \times A) \rightarrow A \cup B$ be a $p$-cyclic contraction mapping. Then $T$ has a unique coupled best proximity point.
\end{theorem}
Gupta and Rohilla \cite{8} introduced the notion of $p$-cyclic nonexpansive mappings and established the existence of coupled best proximity point of such mappings.
\begin{definition}
\emph{\cite[Definition 2.7]{8} Let $A$ and $B$ be nonempty subsets of a Banach space $X$. A mapping $T :(A \times B) \cup (B \times A) \rightarrow A \cup B$ is called a $p$-cyclic nonexpansive mapping if it satisfies the following conditions:}

 \emph{(i) $T(A,B) \subseteq B$ and $T(B,A) \subseteq A$},
 
\emph{(ii) $\Vert T(x_{1},y_{1})-T(x_{2},y_{2}) \Vert \leq \Vert (x_{1},y_{1}) - (x_{2},y_{2}) \Vert$.}
\end{definition}
\begin{theorem}
\cite[Theorem 2.11]{8} Let $A$ and $B$ be nonempty closed, bounded  and convex subsets of a uniformly convex Banach space $X$ such that $A_0 \times B_0$ is compact. Let $T :(A \times B) \cup (B \times A) \rightarrow A \cup B$ be a $p$-cyclic nonexpansive mapping. Then $T$ has a coupled best proximity point.
\end{theorem}
In this paper, we introduce the concept of $p$-cyclic $\phi$ contraction and $p$-cyclic Kannan nonexpansive mappings. Motivated by Al-Thagafi and Shahzad \cite{15}, we extend the notion of $p$-cyclic contraction mapping  introduced by Gupta and Rohilla \cite{8} and introduced the notion of $p$-cyclic $\phi$ contraction mappings. We extend the notion of cyclic Kannan nonexpansive mappings introduced by Gabeleh \cite{17}. The main objective is to establish the existence of coupled best proximity points of these mappings.  The main objective of the paper is to formulate necessary conditions for the existence of coupled best proximity point of such mappings. An example is also provided to support the results obtained.
   
\section{Main Results}
   In this section we obtain coupled best proximity point results of $p$-cyclic $\phi$ contraction and $p$-cyclic Kannan nonexpansive mappings. To establish our results, we intoduce the following new class of mappings: 
     \begin{definition}
 \emph{Let $A$ and $B$ be nonempty subsets of a Banach space $X$. A mapping $ T:(A \times B) \cup (B \times A) \rightarrow A \cup B$ called a $p$-cyclic $\phi$-contraction mapping if the following conditions are satisfied:}
   
 (i)  $T(A,B) \subseteq B$ and $T(B,A) \subseteq A$,
 
(ii) $ \Vert T(x_{n},y_{n})-T(x_{n+1},y_{n+1})\Vert  \leq  \Vert(x_{n},y_{n}) - (x_{n+1},y_{n+1})\Vert-\phi( \Vert(x_{n},y_{n}) - (x_{n+1},y_{n+1})\Vert) +\phi(dist(A,B)$, \emph{where $\phi :[0, \infty) \rightarrow [0, \infty)$ is a strictly increasing map.} 
   \end{definition}
   If we take $\phi(t)=(1-\lambda)t$, where $\lambda \in (0,1)$ in the above definition, then it reduces to $p$-cyclic contraction mapping \cite[Definition 2.1]{8}. 
    
Now, we give the following useful results that will be used to prove the existence of coupled best proximity point of a $p$-cyclic $\phi$ contraction mapping.  
 \begin{lemma}\label{lemma 2.2}
 Let $A$ and $B$ be two  nonempty subsets of a normed space $X$. Let $ T:(A \times B) \cup (B \times A) \rightarrow A \cup B$ be a $p$-cyclic $\phi$-contraction mapping. For $(x_{0},y_{0})  \in A \times B $, define $ x_{n} = T(x_{n-1},y_{n-1})$ and $y_{n} = T( y_{n-1},x_{n-1})$ for each  $n \in \mathbb{N}$. Then 
 
   (i)  $ \phi(dist(A,B)) \leq  \phi(\Vert(x_{n},y_n) - (x_{n-1},y_{n-1}) \Vert$,
   
  (ii) $\Vert T(x_{n},y_{n})-T(x_{n-1},y_{n-1}) \Vert \leq 
 \Vert(x_{n},y_{n})-(x_{n-1},y_{n-1})\Vert$,
 
 (iii) $\Vert (x_{n+2},y_{n+2})-(x_{n+1},y_{n+1})\Vert \leq \Vert (x_{n+1},y_{n+1})-(x_{n},y_{n})\Vert$.
\end{lemma}

  \begin{proof}
  (i) We have $dist(A,B) \leq \Vert x_n-x_{n-1} \Vert$ and $dist(A,B) \leq \Vert y_n-y_{n-1} \Vert$. Therefore, $dist(A,B) \leq \Vert (x_n,y_n)-(x_{n-1},y_{n-1}) \Vert$. Since $\phi$ is strictly increasing mapping,  $ \phi (dist(A,B)) \leq \phi (\Vert (x_{n},y_{n})-(x_{n-1},y_{n-1}) \Vert$.

(ii)  Since T is a $\phi$ $p$-cyclic contraction mapping, 
            $ \Vert T(x_{n},y_{n})-T(x_{n-1}-y_{n-1})\Vert \leq \Vert(x_{n},y_{n})-(x_{n-1},y_{n-1})\Vert-\phi(\Vert(x_{n},y_{n})-(x_{n-1},y_{n-1})\Vert)+ \phi(dist(A,B))$. Using (i), it follows that  $ \Vert T(x_{n},y_{n})-T(x_{n-1},y_{n-1})\Vert \leq \Vert(x_{n},y_{n})-(x_{n-1},y_{n-1})\Vert$.       
 
(iii) Using (ii), we have $\Vert x_{n+2}-x_{n+1}\Vert = \Vert T(x_{n+1},y_{n+1})-T(x_n,y_n)\Vert \leq \Vert (x_{n+1},y_{n+1})-(x_n,y_n)\Vert$ and  $\Vert y_{n+2}-y_{n+1}\Vert = \Vert T(y_{n+1},x_{n+1})-T(y_n,x_n)\Vert \leq \Vert (y_{n+1},x_{n+1})-(y_n,x_n)\Vert$. This gives $\Vert (x_{n+2},y_{n+2})-(x_{n+1},y_{n+1})\Vert \leq \Vert (x_{n+1},y_{n+1})-(x_n,y_n)\Vert$. 
  \end{proof}      
\begin{theorem}\label{theorem2.3}
Let $A$ and $B$ be two  nonempty subsets of a normed space $X$. Let $ T: (A \times B) \cup (B \times A) \rightarrow A \cup B$ be a   $p$-cyclic $\phi$-contraction mapping. For $(x_{0},y_{0})  \in A \times B $, define $ x_{n} = T(x_{n-1},y_{n-1})$ and $y_{n} = T( y_{n-1},x_{n-1})$ for each  $n \in \mathbb{N}$. Then $ \Vert (x_{n},y_{n})-(x_{n+1}, y_{n+1})\Vert \rightarrow dist(A,B)$. 
 \end{theorem}
 \begin{proof} 
 Let $ t_{n} = \Vert(x_{n},y_{n})-(x_{n+1},y_{n+1})\Vert$. By Lemma \ref{lemma 2.2}(iii), $\{t_{n}\}$ is a decreasing and bounded sequence. Therefore, $\lim\limits_{n \rightarrow \infty}  t_{n} = t_{0}$ for some $t_{0}\geq dist(A,B)$. If $t_{n_0}=0$ for some $n_{0}\geq 1$, then there is nothing to prove. Therefore, assume that $t_{n} > 0$ for each $n \in \mathbb{N}$. Since T is a $p$-cyclic $\phi$-contraction mapping, 
$$ \Vert(x_{n+1},y_{n+1})-(x_{n+2},y_{n+2})\Vert \leq \Vert (x_{n},y_{n})-(x_{n+1},y_{n+1}) \Vert- \phi(\Vert (x_{n},y_{n})-(x_{n+1},y_{n+1}) \Vert) + \phi(dist(A,B))$$ which gives 
$$\phi(\Vert (x_{n},y_{n})-(x_{n+1},y_{n+1}) \Vert) \leq \Vert (x_{n},y_{n})-(x_{n+1},y_{n+1}) \Vert -\Vert(x_{n+1},y_{n+1})-(x_{n+2},y_{n+2})\Vert + \phi(dist(A,B)).$$ Using Lemma \ref{lemma 2.2}(i), $\phi(dist(A,B))\leq \phi(\Vert (x_{n},y_{n})-(x_{n+1},y_{n+1}) \Vert) \leq \Vert (x_{n},y_{n})-(x_{n+1},y_{n+1}) \Vert -\Vert(x_{n+1},y_{n+1})-(x_{n+2},y_{n+2})\Vert + \phi(dist(A,B))$. This implies  $\lim\limits_{n \rightarrow \infty} \phi(\Vert (x_{n},y_{n})-(x_{n+1},y_{n+1}) \Vert)= \phi(dist(A,B))$. 
   As $dist(A,B) \leq t_{0} \leq \Vert (x_n,y_n)-(x_{n+1},y_{n+1})\Vert$ for each $n\geq1$. As $\phi$ is strictly increasing mapping, $\phi(dist(A,B)) \leq \phi(t_{0}) \leq \phi(\Vert (x_n,y_n)-(x_{n+1},y_{n+1})\Vert)$. Letting $n \rightarrow \infty$, $\phi(dist(A,B)) \leq \phi(t_{0}) \leq \phi(dist(A,B))$. Since $\phi$  is strictly increasing mapping, $t_{0} = dist(A,B)$, which gives $ \Vert (x_{n},y_{n})-(x_{n+1},y_{n+1})\Vert \rightarrow dist(A,B)$. 
\end{proof}
\begin{theorem}\label{theorem2.4}
 Let $A$ and $B$ be two  nonempty subsets of a normed space $X$. Let $ T: (A \times B) \cup (B \times A) \rightarrow A \cup B$ be a   $p$-cyclic $\phi$-contraction mapping. For $(x_{0},y_{0})  \in A \times B $, define $ x_{n} = T(x_{n-1},y_{n-1})$ and $y_{n} = T( y_{n-1},x_{n-1})$ for each  $n \in \mathbb{N}$. If ${(x_{2n},y_{2n})}$ has a convergent subsequence in  $A\times B$, then there exists $(x,y)\in A\times B$ such that 
$$\Vert x-T(x,y) \Vert= dist(A,B)  \thinspace \emph{and} \thinspace \Vert y-T(y,x) \Vert = dist(A,B).$$
\end{theorem}  
\begin{proof}
 Suppose that $\{(x_{2n_k},y_{2n_k})\}$ be a subsequence of $\{ (x_{2n},y_{2n})\}$ converging to $(x,y)\in A\times B$. Consider
 \begin{align*}
     dist(A,B)&\leq\max\{\Vert x-x_{2n_k-1}\Vert, \Vert y-y_{2n_k-1}\Vert\}\\
     &=\Vert (x,y)-(x_{2n_k-1},y_{2n_k-1})\Vert\\
     & \leq \Vert (x,y)-(x_{2n_k},y_{2n_k}) \Vert+\Vert (x_{2n_k},y_{2n_k})-(x_{2n_k-1},y_{2n_k-1})\Vert.
 \end{align*}
Using Theorem \ref{theorem2.3}, $\Vert (x_{2n_k},y_{2n_k})-(x_{2n_k-1},y_{2n_k-1}) \Vert \rightarrow dist(A,B)$ which gives $\Vert (x,y)-(x_{2n_k-1},y_{2n_k-1})\Vert \rightarrow d(A,B)$. Consider
 \begin{align*}
  dist(A,B)&\leq \Vert (x_{2n_k},y_{2n_k})-(T(x,y),T(y,x))\Vert \\
 & = \Vert (T(x_{2n_k-1},y_{2n_k-1}),T((y_{2n_k-1},x_{2n_k-1}))-(T(x,y),T(y,x))\Vert\\
 &= \max\{\Vert T(x_{2n_k-1},y_{2n_k-1})-T(x,y) \Vert,\Vert T(y_{2n_k-1},x_{2n_k-1})-T(y,x) \Vert\}\\
 &\leq \max\{\Vert(x_{2n_k-1},y_{2n_k-1})-(x,y) \Vert (y_{2n_k-1}, x_{2n_k-1})-(y,x) \Vert \}\\
 &= \Vert(x,y)-(x_{2n_k-1}, y_{2n_k-1})\Vert).  
 \end{align*}
This gives $\Vert (x_{2n_k},y_{2n_k})-(T(x,y),T(y,x))\Vert= dist(A,B)$. Since $\Vert. \Vert $ is weakly lower semi-continuous, $ \Vert (x,y)-(T(x,y),T(y,x))\Vert= dist(A,B)$. Consider 
 \begin{align*}
   dist(A,B) \leq \Vert x-T(x,y) \Vert&\leq \max \{ \Vert x-T(x,y) \Vert,  \Vert y-T(y,x) \Vert\}\\
   &= \Vert (x,y)-(T(x,y),T(y,x)) \Vert\\
   &=dist(A,B).
 \end{align*} 
  Therefore, $ \Vert x-T(x,y) \Vert=dist(A,B)$. Similarly, we can prove that $ \Vert y-T(y,x) \Vert=dist(A,B)$.
\end{proof}
\begin{lemma}\label{lemma2.5}
Let $A$ and $B$ be nonempty subsets of a uniformly convex Banach space $X$ such that $A$ is convex. Let $ T:(A \times B) \cup (B \times A) \rightarrow A \cup B$ be a $p$-cyclic $\phi$-contraction mapping. For $(x_{0},y_{0})  \in A \times B $, define $ x_{n} = T(x_{n-1},y_{n-1})$ and $y_{n} = T( y_{n-1},x_{n-1})$ for each  $n \in \mathbb{N}$. Then 
$$ \Vert (x_{2n+2},y_{2n+2}) - (x_{2n},y_{2n})\Vert \rightarrow 0 \thinspace \emph{and} \thinspace \Vert (x_{2n+3},y_{2n+3})-(x_{2n+1},y_{2n+1}) \Vert \rightarrow 0 \thinspace \emph{as} \thinspace n \to \infty.$$
\end{lemma}
\begin{proof}
First we show that $\Vert x_{2n+2}-x_{2n} \Vert \rightarrow 0$ as $n \rightarrow \infty$. On the contrary, assume that there exists $\epsilon_{0}>0$ such that for each $ k\geq 1$, there exits $n_{k}\geq k$ 
\begin{equation}\label{equation2.1}
\Vert (x_{2n_k+2}- x_{2n_k}\Vert \geq \epsilon_{0}. 
\end{equation}
Choose $0 < \gamma < 1 $ such that $\frac{\epsilon_{0}}{\gamma} > dist(A,B)$. Choose $\epsilon $ such that $ 0 < \epsilon < \min\{ \frac{\epsilon_{0}}{\gamma}-dist(A,B), \frac{dist(A,B)\delta(\gamma)}{1-\delta(\gamma)}\} $. By Theorem \ref{theorem2.3},  there exists $N_1$ such that
\begin{equation}\label{equation2.2}
\Vert x_{2n_k+2}-x_{2n_k+1}\Vert < dist(A,B)+ \epsilon \mbox{ for all }n_k\geq N_1. 
\end{equation}
Also, by Theorem \ref{theorem2.3}, there exists $N_{2}$ such that 
\begin{equation}\label{equation2.3}
\Vert x_{2n_k}-x_{2n_k+1}\Vert < dist(A,B) +\epsilon \mbox{ for all }n_k\geq N_2.    
\end{equation} 
Let $ N= \max\{N_{1},N_{2}\}$. Using \ref{equation2.1}-\ref{equation2.3} and uniform convexity of $X$, it follows that
\begin{align*}
\Big\Vert \frac{x_{2n_k+2}+x_{2n_k}}{2}-x_{2n_k+1}\Big\Vert \leq \Big(1- \delta\Big(\frac{\epsilon_{0}}{dist(A,B)+\epsilon}\Big)\Big)(dist(A,B)+\epsilon)
\end{align*}
As $\epsilon<\frac{\epsilon_0}{\gamma}-dist(A,B)$, we have $\gamma <\frac{\epsilon_0}{dist(A,B)+\epsilon}$. As $\delta$ is a strictly increasing function, $\delta(\gamma) <\delta\Big(\frac{\epsilon_0}{dist(A,B)+\epsilon}\Big)$. Therefore, 
\begin{align*}
\Big\Vert \frac{x_{2n_k+2}+x_{2n_k}}{2}-x_{2n_k+1}\Big\Vert \leq (1- \delta(\gamma))(dist(A,B)+\epsilon). 
\end{align*}
As $\epsilon<\frac{dist(A,B)\delta(\gamma)}{1-\delta(\gamma)}$. Therefore,
\begin{align*}
\Big\Vert \frac{x_{2n_k+2}+x_{2n_k}}{2}-x_{2n_k+1}\Big\Vert < (1- \delta(\gamma))dist(A,B)+dist(A,B)\delta(\gamma)
\end{align*}
which gives $\Big\Vert \frac{x_{2n_k+2}+x_{2n_k}}{2}-x_{2n_k+1}\Big\Vert < dist(A,B)$, a contradiction. Since $A$ is convex and $x_{2n_k+2}, x_{2n_k} \in A$, $\frac{x_{2n_k+2}}{x_{2n_k}}{2} \in A$. Also, $x_{2n_k+1} \in B$, $dist(A,B) \leq \Big\Vert \frac{x_{2n_k+2}+x_{2n_k}}{2}-x_{2n_k+1}\Big\Vert$. Thus, $\Vert x_{2n+2}-x_{2n} \Vert \rightarrow 0$ as $n \rightarrow \infty$. Similarly, we prove that $\Vert x_{2n+3}-x_{2n+1}\Vert \rightarrow 0$,  $\Vert y_{2n+2}-y_{2n} \Vert \rightarrow 0$ and $\Vert y_{2n+3}-y_{2n+1}\Vert \rightarrow 0$ as $n \rightarrow \infty$. Hence, $\Vert (x_{2n+2},y_{2n+2})-(x_{2n},y_{2n})\Vert \rightarrow 0$ and  $\Vert (x_{2n+3},y_{2n+3})-(x_{2n+1},y_{2n+1})\Vert \rightarrow 0$ as $n \rightarrow \infty$.
\end{proof}
\begin{theorem}\label{theorem2.6}
 Let $A$ and $B$ be nonempty subsets of a uniformly convex Banach space $X$ such that $A$ is convex. Let $ T: (A \times B) \cup (B \times A) \rightarrow A \cup B $ be a $p$-cyclic $\phi$-contraction mapping. For $(x_{0},y_{0})  \in A \times B $, define $ x_{n} = T(x_{n-1},y_{n-1})$ and $y_{n} = T( y_{n-1},x_{n-1})$ for each  $n \in \mathbb{N}$. Then for each $\epsilon > 0$, there is a positive integers $N$ such that  for all  $m > n \geq N$
  $$ \Vert (x_{2m}, y_{2m}) - (x_{2n+1},y_{2n+1}) \Vert < dist(A,B)+ \epsilon.$$
\end{theorem}
\begin{proof}
On the contrary, assume that there exits $\epsilon_{0}$ such that for each $k\geq1$, there is $m_{k} > n_{k}\geq k$ satisfying
 \begin{equation}\label{equation2.4}
   \Vert (x_{{2m_k}}, y_{{2m_k}}) - (x_{{2n_k+1}}, y_{{2n_k+}}) \Vert \geq dist(A,B)+\epsilon_{0}  
 \end{equation}
  and
  \begin{equation}\label{equation 2.5}
    \Vert(x_{2(m_k-1)}, y_{2(m_k-1)})-(x_{2n_k+1},y_{2n_k+)}) \Vert < dist(A,B) + \epsilon_{0}.   
  \end{equation}
 Using (\ref{equation2.4}) and (\ref{equation 2.5}), we have
 \begin{align*}
  dist(A,B)+\epsilon_{0}& \leq \Vert (x_{2m_k},y_{2m_k})-(x_{2n_k+1}, y_{2n_k+1}) \Vert \\
  & \leq \Vert (x_{2m_k},y_{2m_k})-(x_{2(m_k-1)},y_{2(m_k-1)})\Vert + \Vert (x_{2(m_k-1)},y_{2(m_k-1)})- (x_{2n_k+1},y_{2n_k+1}) \Vert \\
  & < \Vert (x_{2m_k},y_{2m_k})-(x_{2(m_k-1)},y_{2(m_k-1)}) \Vert+dist(A,B) +\epsilon_{0}.
 \end{align*}
Letting $ k \rightarrow \infty$ and using Lemma  \ref{lemma2.5}, we get 
\begin{equation}\label{equation2.6}
        \lim_{k \to \infty} \Vert (x_{2m_k}, y_{2m_k})-(x_{2n_k+1}, y_{2n_k+1}) \Vert =dist(A,B)+ \epsilon_{0}. 
    \end{equation}    
Using \ref{lemma 2.2}(iii) and $p$-cyclic $\phi$-contraction property of $T$, we have
 \begin{align*}
  \Vert(x_{2m_k},y_{2m_k})-(x_{2n_k+1},y_{2n_k+1})\Vert & \leq \Vert (x_{2m_k},y_{2m_k})-(x_{2m_k+2},y_{2m_k+2}) \Vert +\Vert (x_{2m_k+2},y_{2m_k+2})\\
  & \quad -(x_{2n_k+3},y_{2n_k+3})\Vert + \Vert (x_{2n_k+3},y_{2n_k+3})-(x_{2n_k+1},y_{2n_k+1})\Vert\\
  & \leq \Vert (x_{2m_k},y_{2m_k})-(x_{2m_k+2},y_{2m_k+2}) \Vert + \Vert (x_{2m_k+1},y_{2m_k+1})\\
  & \quad -(x_{2n_k+2},y_{2n_k+2})\Vert +\Vert (x_{2n_k+3},y_{2n_k+3})-(x_{2n_k+1},y_{2n_k+1}) \Vert\\
  & = \Vert (x_{2m_k},y_{2m_k})-(x_{2m_k+2},y_{2m_k+2}) \Vert + \Vert (T(x_{2m_k},y_{2m_k}), T(y_{2m_k},\\
  & \quad \quad x_{2m_k}))-(T(x_{2n_k+1},y_{2n_k+1}), T(y_{2n_k+1},x_{2n_k+1}))\Vert +\Vert (x_{2n_k+3},\\
  & \quad \quad y_{2n_k+3}) -(x_{2n_k+1},y_{2n_k+1}) \Vert\\
  & \leq \Vert (x_{2m_k},y_{2m_k})-(x_{2m_k+2},y_{2m_k+2}) \Vert + \Vert (x_{2m_k},y_{2m_k})-(x_{2n_k+1},\\
  & \quad \quad  y_{2n_k+})\Vert- \phi(\Vert (x_{2m_k},y_{2m_k})-(x_{2n_k+1},y_{2n_k+}))\Vert  \\ & \quad + \phi(dist(A,B))+\Vert (x_{2n_k+3},y_{2n_k+3})
   -(x_{2n_k+1},y_{2n_k+1}) \Vert. 
 \end{align*}
  Letting $ k \rightarrow \infty $, using (\ref{equation2.4}), (\ref{equation2.6}) and Lemma \ref{lemma2.5}, we obtain
\begin{align*}
        dist(A,B)+ \epsilon_{0} & \leq dist(A,B)+\epsilon_{0}- \lim_{k \to \infty} \phi(\Vert (x_{2m_k},y_{2m_k})-(x_{2n_k+1},y_{2n_k+})\Vert)+ \phi(dist(A,B))\\
 & \leq dist(A,B) +\epsilon_{0}
\end{align*}
which implies that 
\begin{equation}\label{equation2.7}
   \lim_{k \to \infty} \phi(\Vert (x_{2m_k},y_{2m_k})-(x_{2n_k+1},y_{2n_k+1})\Vert)= \phi(dist(A,B)).   
\end{equation}
Since $ \phi$ is strictly increasing, using (\ref{equation2.4}) and (\ref{equation2.7}) it follows that
\begin{align*}
\phi(dist(A,B)+\epsilon_{0}) & \leq \lim_{k \to \infty} \phi(\Vert (x_{2m_k},y_{2m_k})-(x_{2n_k+1},y_{2n_k+})\Vert)\\
& = \phi(dist(A,B))\\
& < \phi(dist(A,B)+ \epsilon_{0}),
\end{align*}
a contradiction. Hence, for each $\epsilon > 0$, there is a positive integers $N$ such that  for all  $m > n \geq N$, we have $ \Vert (x_{2m}, y_{2m}) - (x_{2n+1},y_{2n+1}) \Vert < dist(A,B)+ \epsilon.$    
\end{proof}
\begin{theorem}\label{theorem2.7}
Let $A$ and $B$ be nonempty closed subsets of a uniformly convex Banach space $X$. Let $T:(A \times B) \cup (B \times A) \rightarrow A \cup B $ be a $p$-cyclic $\phi$-contraction mapping. For $(x_{0},y_{0})  \in A \times B $, define $ x_{n} = T(x_{n-1},y_{n-1})$ and $y_{n} = T( y_{n-1},x_{n-1})$ for each  $n \in \mathbb{N}$. If $dist(A,B)=0$, then $T$ has a unique coupled fixed point $(x,y) \in A \times B$ and $x_n \rightarrow x$ and $y_n \rightarrow y$ as $n \rightarrow \infty$. 
\end{theorem}
\begin{proof}
Let $\epsilon>0$ be given. By Theorem \ref{theorem2.3}, there exists $N_1$ such that
\begin{equation}\label{equation2.8}
\Vert (x_n,y_n)-(x_{n+1},y_{n+1}) \Vert < \epsilon  \mbox{ for all }  n\geq N_1.
\end{equation}
By Theorem \ref{theorem2.6}, there exists $N_2$ such that
\begin{equation}\label{equation2.9}
\Vert (x_{2m},y_{2m})-(x_{2n+1},y_{2n+1})\Vert < \epsilon \mbox{ for all }  m>n\geq N_2.
\end{equation}
Let $N=\max\{N_1,N_2\}$. For $m>n\geq N$, using (\ref{equation2.8}) and (\ref{equation2.9}), we have
\begin{align*}
\Vert (x_{2m},y_{2m})-(x_{2n},y_{2n})\Vert & \leq \Vert (x_{2m},y_{2m})-(x_{2n+1},y_{2n+1})\Vert+\Vert (x_{2n+1},y_{2n+1})-(x_{2n},y_{2n})\Vert\\
&<  2 \epsilon.
\end{align*}
This gives $\{(x_{2n},y_{2n})$ is a Cauchy sequence. Since $X$ is complete and $A$ and $B$ are closed subsets of $X$, $A$ and $B$ are complete. This implies that $\{(x_{2n},y_{2n})$ is convergent in $A \times B$. By Theorem \ref{theorem2.4}, there exists $(x,y) \in A \times B$ such that $\Vert x-T(x,y) \Vert=0$ and $\Vert y-T(y,x)\Vert=0$ which gives $T(x,y)=x$ and $T(y,x)=y$. Therefore, $(x,y)$ is a coupled fixed point of $T$. 

To establish the uniqueness of coupled fixed point of $T$, let $(x^*,y^*)$ be another coupled fixed point of $T$. Consider
\begin{align*}
\Vert x-x^* \Vert&=\Vert T(x,y)-T(x^*,y^*)\Vert \\
& \leq \Vert (x,y)-(x^*,y^*) \Vert -\phi(\Vert (x,y)-(x^*,y^*) \Vert)+\phi(0).
\end{align*}
Similarly, we have 
$$\Vert y-y^* \Vert \leq \Vert (x,y)-(x^*,y^*) \Vert -\phi(\Vert (x,y)-(x^*,y^*) \Vert)+\phi(0).$$
Therefore, $\Vert (x,y)-(x^*,y^*) \Vert \leq \Vert (x,y)-(x^*,y^*) \Vert -\phi(\Vert (x,y)-(x^*,y^*) \Vert)+\phi(0)$. This gives
\begin{equation}\label{equation2.10}
\phi(\Vert (x,y)-(x^*,y^*) \Vert) \leq \phi(0). 
\end{equation}
Since $\Vert x-x^*\Vert>0$ and $\Vert y-y^* \Vert>0$, $\Vert (x,y)-(x^*,y^*)\Vert >0$. As $\phi$ is a strictly incresing mapping,
\begin{equation}\label{equation2.11}
\phi(0)< \phi(\Vert (x,y)-(x^*,y^*) \Vert)
\end{equation}
Using (\ref{equation2.10}) and (\ref{equation2.11}), it follows that $\phi(0)<\phi(0)$, a contradiction. Hence, the coupled fixed point of $T$ is unique. 
\end{proof}
\begin{theorem}\label{theorem2.8}
Let $A$ and $B$ be nonempty closed and convex subsets of a uniformly convex Banach space $X$. Let $T: (A \times B) \cup (B \times A) \rightarrow A \cup B $ be a $ p$-cyclic $\phi$-contraction mapping. For $(x_{0},y_{0})  \in A \times B $, define $ x_{n} = T(x_{n-1},y_{n-1})$ and $y_{n} = T( y_{n-1},x_{n-1})$ for each  $n \in \mathbb{N}$.  Then sequence $\{(x_{2n},y_{2n}\})$ and $\{(x_{2n+1},y_{2n+1})\}$ are Cauchy sequences in $A \times B$ and $B \times A$, respectively.
\end{theorem}
\begin{proof}
If $dist(A,B)=0$, then the result follows from Theorem \ref{theorem2.7}. Therefore, suppose that $dist(A,B)>0$. We will prove that $\{(x_{2n},y_{2n})\}$ is a Cauchy sequence in $A \times B$. On the contrary, assume that $\{x_{2n}\}$ is not a Cauchy sequence in $A$. Then there exists $\epsilon_{0}>0$ such that $k \geq 1$, there exist $m_{k}> n_{k}\geq k$ such that
\begin{equation}\label{equation2.12}
\Vert x_{{2m_k}}- x_{{2n_k}}\Vert \geq \epsilon_{0}. 
\end{equation}
Choose  $0< \gamma<1$ such that $\frac{\epsilon_0}{\gamma} > dist(A,B)$ and choose $\epsilon$ such that
$$0< \epsilon <  \min \Big \{ \frac{\epsilon_{0}}{\gamma}-dist(A,B),\frac{dist(A,B)\delta(\gamma)}{1-\delta(\gamma)} \Big\}.$$ 
By Theorem \ref{theorem2.3}, there exits $N_{1}$ such that
 \begin{equation}\label{equation2.13}
  \Vert x_{{2n_k}}- x_{{2n_k+1}} \Vert < dist(A,B)+\epsilon \mbox{ for all } n_k \geq N_1.
  \end{equation}
By Theorem \ref{theorem2.6}, there exits $N_{2}$ such that
\begin{equation}\label{equation2.14}
\Vert x_{{2m_k}} - x_{{2n_k+1}} \Vert < dist(A,B)+\epsilon \mbox{ for all } m_{k} > n_{k} \geq N_{2}.
\end{equation}
Let $ N = \max\{N_{1},N_{2}\}$. Using (\ref{equation2.12})-(\ref{equation2.14}) and the uniform convexity of $X$, we have 
$$\Big \Vert \frac{x_{{2m_k}}+x_{{2n_k}}}{2}-x_{{2n_k+1}} \Big\Vert \leq \Big(1-\delta\Big(\frac{\epsilon_0}{dist(A,B)+\epsilon}\Big)\Big)(dist(A,B)+\epsilon) $$
As $\epsilon<\frac{\epsilon_0}{\gamma}-dist(A,B)$, we have $\gamma<\frac{\epsilon_0}{\epsilon+dist(A,B)}$. As $\delta$ is a strictly increasing function, $\delta(\gamma) <\delta\Big(\frac{\epsilon_0}{dist(A,B)+\epsilon}\Big)$. This gives $$\Big\Vert \frac{x_{{2m_k}}+x_{{2m_k}}}{2}- x_{{2n_k+1}}  \Big\Vert \leq (1-\delta(\gamma))(dist(A,B)+\epsilon).$$
 As $\epsilon<\frac{dist(A,B)\delta(\gamma)}{1-\delta(\gamma)}$,  
$$ \Big \Vert \frac{x_{{2m_k}}+x_{{2n_k}}}{2}-x_{{2n_k+1}} \Big\Vert  < (1-\delta(\gamma))dist(A,B)+dist(A,B)\delta(\gamma)$$
which gives $\Big \Vert \frac{x_{{2m_k}}+x_{{2n_k}}}{2}-x_{{2n_k+1}} \Big\Vert  < dist(A,B)$, a contradiction. Similarly, we can prove that $\{y_{2n+1}\}$ is a Cauchy sequence in $A$ and $\{y_{2n}\}$ and $\{x_{2n+1}\}$ are Cauchy sequences in $B$. 
\end{proof} 
The following lemma will be used in the sequel to prove the existence of coupled best proximity point of a a $p$-cyclic $\phi$-contraction mapping in the setting of a uniformly convex Banach space.
\begin{lemma}\label{lemma2.9}
\cite[Lemma 2.5]{8} Let $A$ and $B$ be nonempty closed and convex subset of a uniformly convex Banach space $X$. Let $\{(x_{n},y_{n})\}$ and $(w_{n},z_{n})\}$ be sequences in $A \times B$ and $(u_{n},v_{n})$ be a sequence in $B \times A$ satisfying

(i) $ \Vert (x_{n},y_{n})- (u_{n},v_{n}) \Vert \rightarrow dist(A,B)$,

(ii)  $ \Vert (w_{n},z_{n}) - (u_{n},v_{n}) \Vert \rightarrow dist(A,B)$,\\
then $ \Vert (x_{n},y_{n}) - (w_{n},z_{n}) \Vert \rightarrow 0$. 
\end{lemma}
\begin{theorem}\label{theorem2.10}
Let $A$ and $B$ be nonempty closed and convex subsets of a uniformly convex Banach space $X$. Let $ T:(A \times B) \cup (B \times A) \rightarrow A \cup B$ be a $p$-cyclic $\phi$-contraction mapping. For $(x_{0},y_{0})  \in A \times B $, define $ x_{n} = T(x_{n-1},y_{n-1})$ and $y_{n} = T( y_{n-1},x_{n-1})$ for each  $n \in \mathbb{N}$.  
Then $T$ has a unique coupled best proximity point.
\end{theorem}
\begin{proof}
Using Theorem \ref{theorem2.8}, we deduce that $\{x_{2n},y_{2n})\}$ is a Cauchy sequence in $A \times B$. By using Theorem \ref{theorem2.4}, there exists $(x,y) \in A \times B$ such that 
$$\Vert x-T (x, y) \Vert = dist(A, B) \mbox{ and }\Vert y- T (y, x) \Vert = dist(A, B).$$
Therefore, $\Vert (x,y)-(T(x,y),T(y,x))\Vert=\max\{\Vert x-T (x, y) \Vert,\Vert y- T (y, x) \Vert\}=dist(A,B)$. To establish the uniqueness of coupled best proximity point of $T$, let $(x^*,y^*)$ be another coupled best proximity point of T. Consider
\begin{align*}
  \Vert T(T(x,y),T(y,x))-T(x,y) \Vert & \leq \Vert(T(x,y),T(y,x))-(x,y) \Vert -\phi(\Vert (T(x,y),T(y,x))-(x,y) \Vert \\
  & \quad +\phi(dist(A,B))\\
   & = dist(A,B)- \phi (dist(A,B))+  \phi(dist(A,B))\\\end{align*}
This gives $\Vert T(T(x,y),T(y,x))-T(x,y) \Vert =dist(A,B).$
Similarly, we prove that $\Vert T(T(y,x),T(x,y))-T(y,x) \Vert = dist(A,B)$. This implies that
\begin{align*}
\Vert (T(T(x,y),T(y,x)),T(T(y,x),T(x,y)))-(T(x,y),T(y,x)) \Vert&= \max \{ \Vert T(T(x,y),T(y,x))-T(x,y) \Vert,\\ 
&\quad \quad \Vert T(T(y,x),T(x,y))-T(y,x) \Vert \}\\
&= dist(A,B). 
\end{align*}
Using Lemma \ref{lemma2.9}, it follows that $\Vert (x,y)-(T(T(x,y),T(y,x)),T(T(y,x),T(x,y)))\Vert = 0$. This gives
  $$T(T(x,y),T(y,x))=x \quad \mbox{ and } \quad T(T(y,x),T(x,y)) = y.$$
Similarly, we show that 
$$T(T(x^*,y^*),T(y^*,x^*))=x^* \quad \mbox{ and } \quad T(T(y^*,x^*),T(x^*,y^*)) = y^*.$$   
  Conisder
\begin{align*}
  \Vert T(x,y)-x^* \Vert & = \Vert T(x,y)- T(T( x^*,y^*),T(y^*,x^*)) \Vert\\
&\leq \Vert (x,y)- (T( x^*,y^*),T(y^*,x^*)) \Vert - \phi (\Vert (x,y)- (T( x^*,y^*),T(y^*,x^*)) \Vert)+\phi(dist(A,B)).
\end{align*}
Since $dist(A,B) \leq \Vert (x,y)-(T(x^*,y^*),T(y^*,x^*)) \Vert$ and $\phi$ is a strictly increasing mapping, $\Vert T(x,y)-x^* \Vert \leq \Vert (x,y)- (T( x^*,y^*),T(y^*,x^*))\Vert$. Similarly, we show that $\Vert T(y,x)-y^* \Vert \leq \Vert (y,x)- (T(y^*,x^*),T(x^*,y^*))$. This implies that
\begin{equation}\label{equation2.15}
\Vert (T(x,y),T(y,x))- (x^*,y^*)\Vert \leq \Vert (x,y)- (T( x^*,y^*),T(y^*,x^*)) \Vert.
\end{equation}
Similarly,
\begin{equation}\label{equation2.16}
\Vert (T(x^*,y^*),T(y^*,x^*))- (x,y)\Vert \leq \Vert (x^*,y^*)- (T(x,y),T(y,x)) \Vert.
\end{equation}
From (\ref{equation2.15}) and (\ref{equation2.16}), we conclude that
$$\Vert (x,y)- (T( x^*,y^*),T(y^*,x^*)) \Vert=\Vert (x^*,y^*)- (T(x,y),T(y,x)) \Vert.$$
If $\Vert (x,y)- (T( x^*,y^*),T(y^*,x^*)) \Vert=dist(A,B)$, then as $\Vert (x^*,y^*)- (T(x^*,y^*),T(y^*,x^*)) \Vert=dist(A,B)$. Using Lemma \ref{lemma2.9}, $\Vert (x,y)=(x^*,y^*)\Vert=0$ which implies that $(x,y)=(x^*,y^*)$. 
If $\Vert (x,y)- (T( x^*,y^*),T(y^*,x^*)) \Vert>dist(A,B)$, then $\Vert (x^*,y^*)- (T(x,y),T(y,x)) \Vert>dist(A,B)$.  Consider
\begin{align*}
\Vert(x^*,y^*)-(T(x,y),T(y,x))\Vert&=\max\{\Vert x^*-T(x,y)\Vert, \Vert y^*-T(y,x)\Vert\}\\
& = \max \{ \Vert T(T(x^*,y^*),T(y^*,x^*))-T(x,y)\Vert, \Vert T(T(y^*,x^*),T(x^*,y^*))\\
&\quad -T(y,x)\Vert\}\\
& \leq \Vert (T(y^*,x^*),T(x^*,y^*)-(x,y)\Vert - \phi(\Vert (T(y^*,x^*),T(x^*,y^*)-(x,y)\Vert)\\
& \quad +\phi(dist(A,B))\\
& < \Vert (T(y^*,x^*),T(x^*,y^*)-(x,y)\Vert,
\end{align*}
a contradiction.  Hence, $T$ has a unique coupled best proximity point.
\end{proof}
We observe that \cite[Theorem 2.6]{8} can be easily deduced from Theorem \ref{theorem2.10}.
\begin{corollary}
Let $A$ and $B$ be nonempty closed and convex subsets of a uniformly convex Banach space $X$. Let $ T:(A \times B) \cup (B \times A) \rightarrow A \cup B$ be a $p$-cyclic contraction mapping. For $(x_{0},y_{0})  \in A \times B $, define $ x_{n} = T(x_{n-1},y_{n-1})$ and $y_{n} = T( y_{n-1},x_{n-1})$ for each  $n \in \mathbb{N}$.  
Then $T$ has a unique coupled best proximity point.
\end{corollary}
\begin{proof}
Define $\phi:[0,\infty) \rightarrow [0,\infty)$ by $\phi(t)=(1-\lambda)t$, where $\lambda\in (0,1)$. Then the desired result follows from Theorem \ref{theorem2.10}. 
\end{proof}
Now, we introduce a new class of mappings.  
\begin{definition}
Let $(A,B)$ be a nonempty pair in a normed space $X$. A mapping $ T: (A \times B)\cup (B \times A) \rightarrow A \cup B$ is said to be $p$-cyclic Kannan nonexpansive mapping if

(i) $T(A,B) \subseteq B$ and $T(B,A) \subseteq A$,
 
(ii) $\Vert T(x_{1},y_{1}) - T(x_{2},y_{2}) \Vert \leq \frac{1}{2} [\Vert (x_{1},y_{1}) -  (T(x_{1},y_{1}),  T(y_{1},x_{1}))\Vert + \Vert  (x_{2},y_{2}) - (T(x_{2},y_{2}), (T(y_{2},x_{2})) \Vert$.      \end{definition}
\begin{theorem}\label{theorem2.12}
 Let $(A,B)$ be a nonempty, weakly compact and convex pairs subsets of a normed linear space $X$. Let $T: (A \times B)\cup (B \times A) \rightarrow A \cup B$ be a $p$-cyclic Kannan nonexpansive mapping. Suppose that
$$\Vert (T(x,y), T(y,x)) - (T(T(x,y), T(y,x)), T(T(y,x), T(x,y))) \Vert < \Vert (x,y) - (T(x,y), T(y,x)) \Vert$$ with $dist(A,B)< \Vert (x,y) - (T(x,y), T(y,x)) \Vert$. Then $T$ has a coupled best proximity point.
\end{theorem}
\begin{proof}
Let $\Gamma$ denote the collection of all nonempty, weakly compact and convex pairs $(G,H)$ which are subsets of $(A,B)$ and $T$ is $p$-cyclic on $(G \times H) \cup (H \times G)$.   Since $(A,B) \in \Gamma$, $\Gamma$ is nonempty. By using Zorn's lemma $\Gamma$ has a minimal elements say $(E_{1}, E_{2})$ with respect to reverse inclusion relation. Suppose $(x,y) \in E_{1} \times E_{2}$ and $ \Vert(x,y)-(T(x,y),T(y,x))\Vert = r$. If $r = dist(A,B)$, then $(x,y)$ is a coupled best proximity points of $T$. So, assume that $r > dist(A,B)$. Let 
$$W_{1} = \{(u_1,u_2)\in E_{1} \times E_{2} : \Vert (u_1,u_2)-(T(u_1,u_2),T(u_2,u_1)) \Vert \leq r \}.$$ As $\Vert (x,y)-(T(x,y),T(y,x)) \Vert= r$, $W_{1}$ is nonempty.  If $ \tilde{y}=T(x,y)$, then $\tilde{y}\in E_{2}$ and  $ \tilde{x}=T(y,x)$, then $\tilde{x} \in E_{1}$. Consider 
\begin{align*}
 \Vert( \tilde{y},\tilde{x}) -(T(\tilde{y},\tilde{x}),T(\tilde{x},\overline{y})) \Vert & \
    = \Vert (T(x,y),T(y,x)) - (T(T(x,y),T(y,x)), T(T(y,x),T(x,y)) \Vert \\
& = \max \{\Vert T(x,y) - T(T(x,y), T(y,x)) \Vert, \Vert T(y,x) - T(T(y,x),T(x,y)) \Vert \}.
\end{align*}
Consider 
\begin{align*}
\Vert T(x,y) - T(T(x,y),T(y,x)) \Vert & \leq \frac{1}{2} \Big [\Vert (x,y) - (T(x,y), T(y,x)) \Vert  + \Vert (T(x,y), T(y,x))\\
& \quad  - T(T(x,y), T(y,x)),T(T(y,x), T(x,y)) \Vert \Big]. 
 \end{align*} 
Also, 
\begin{align*}
\Vert T(y,x)-T(T(y,x),T(x,y)) \Vert & \leq \frac{1}{2} \Big [\Vert (y,x) - (T(y,x), T(x,y)) \Vert+ \Vert (T(y,x), T(x,y))  \\ & \quad - (T(T(y,x), T(x,y)),T(T(x,y), T(y,x)) \Vert \Big ]. 
\end{align*}
This implies that
\begin{align*}
\Vert( \tilde{y},\tilde{x}) -(T(\tilde{y},\tilde{x}),T(\tilde{x},\overline{y})) \Vert & \leq \frac{1}{2} \Big [\Vert (x,y)-(T(x,y),T(y,x)) \Vert + \Vert (T(x,y),T(y,x))\\
& \quad-(T(T(x,y),T(y,x)),T(T(y,x),T(x,y)))\Vert\\
   & \leq \frac{1}{2} \Big [r+ \Vert (\tilde{y},\tilde{x})-(T(\tilde{y}, \tilde{x}),T(\tilde{x}, \tilde{y})) \Vert\Big].
\end{align*}  
This gives 
\begin{equation}\label{equation2.17}
\Vert( \tilde{y},\tilde{x}) -(T(\tilde{y},\tilde{x}),T(\tilde{x},\overline{y})) \Vert
 \leq r.
 \end{equation}  Let 
$$W_{2} =  \{(v_{1}, v_{2}) \in E_{2}\times E_{1} : \Vert (v_{1}, v_{2}) -(T(v_{1}, v_{2}), T(v_{2}, v_{1})) \leq r \Vert \}.$$
As $ (\overline{y},\overline{x}) \in W_{2}$, $W_{2}$ is nonempty. Let $ V_{1} = \overline{\mbox{con}}(T(W_{2}))$ and   $V_{2} = \overline{\mbox{con}}(T(W_{1}))$. Now, we prove that $ V_{1}\times V_{2} \subseteq W_{1}$. Let $ (u,v)\in V_{1}\times V_{2}$. Then $u \in \overline{\mbox{con}}(T(W_{2})) $ and $ v \in \overline{\mbox{con}}(T(W_{1}))$. Since $u \in \overline{\mbox{con}}(T(W_{2})$, there exists  $(c_{1},d_{1}),(c_{2},d_{2}), \ldots, (c_{n},d_{n}) \in W_{2}$ and $t_{1},t_{2}, \ldots ,t_{n} \in [0,1]$ with  $\sum\limits_{i=1}^n t_{i} =1$ and  
\begin{equation*}
    \Vert u - \sum\limits_{i=1}^n t_{i} T(c_{i}, d_{i}) \Vert <  \epsilon.
\end{equation*}
    Since $ v \in \overline{\mbox{con}}(T(W_{1}))$, there exists  $(e_{1},f_{1}),(e_{2},f_{2}),\ldots,(e_{m},d_{m}) \in W_{1}$ and $s_{1}, s_{2},\ldots , s_{m} \in [0, 1] $ with $ \sum\limits_{i=1}^m s_{i}=1$ and 
  \begin{equation*}
    \Vert v - \sum\limits_{i=1}^m s_{i} T(e_{i}, f_{i}) \Vert \leq \epsilon.  
  \end{equation*}
 Consider 
\begin{align*}
 \Vert u-T(u,v) \Vert & \leq \Big\Vert u -\sum\limits_{i=1}^n t_{i} T(c_{i},d_{i}) \Big\Vert + \Big\Vert \sum\limits_{i=1}^n t_{i} T(c_{i},d_{i}) - T(u,v) \Big\Vert \\
 & < \epsilon + \Big\Vert \sum\limits_{i=1}^n t_{i} T(c_{i},d_{i}) - \sum\limits_{i=1}^n t_{i} T(u,v) \Big\Vert\\
 & \leq \epsilon +  \sum\limits_{i=1}^n t_{i} \Vert T(c_{i},d_{i}) - T(u,v)\Vert\\
 & \leq \epsilon + \sum\limits_{i=1}^n \frac{t_{i}}{2} [\Vert (c_{i}, d_{i}) - (T(c_{i}, d_{i}), T(d_{i}, c_{i})) \Vert + \Vert (u,v) - (T(u,v),T(v,u)) \Vert]\\
 & \leq \epsilon + \frac{1}{2} \sum\limits_{i=1}^n t_{i} (r+ \Vert (u,v) - (T(u,v), T(v,u)) \Vert\\
 &= \epsilon + \frac{r}{2} + \frac{1}{2} \Vert (u,v) - (T(u,v), T(v,u)) \Vert. 
\end{align*}
Similarly, we show that 
\begin{equation*}
   \Vert v-T(v,u) \Vert < \epsilon + \frac{r}{2}
+ \frac{1}{2} \Vert (u,v) - (T(u,v), T(v,u)) \Vert. 
\end{equation*}
 Therefore,
  $$ \Vert (u,v) - (T(u,v), T(v,u)) \Vert < \epsilon + \frac{r}{2}
+ \frac{1}{2} \Vert (u,v) - (T(u,v), T(v,u)) \Vert$$ which implies that
$$\Vert (u,v) - (T(u,v), T(v,u)) \Vert < 2\epsilon + 2.$$ Since $\epsilon$ is arbitrary, $\Vert (u,v)-(T(u,v), T(v,u)) \Vert < r$ which gives $(u,v) \in W_{1}$. Therefore, $V_{1} \times V_{2} \subseteq W_{1}$. Similarly, we can show $V_{2} \times V_{1} \subseteq W_{2}$.
 Since $ V_{1} \times V_{2} \subseteq W_{1}$, $T(V_{1} \times  V_{2}) \subseteq T(W_{1}) \subseteq \overline{\mbox{cov}} (T(W_{1})) = V_{2}$. Also  $ V_{2} \times V_{1} \subseteq W_{2}$, $ T(V_{2} \times V_{1}) \subseteq T(W_{2}) \subseteq \overline{\mbox{con}} (T(W_{2})= V_{1}$. It follows that $T$ is $p$-cyclic on $(V_{1} \times V_{2}) \cup (V_{2} \times V_{1})$ and  $(E_{1}, E_{2}) $ is minimal element which gives that $V_{1} = E_{1}$ and $V_{2} = E_{2}$. Therefore, $ E_{1} \times E_{2} = V_{1} \times V_{2} \subseteq W_{1} \subseteq E_{1} \times E_{2}$ which implies that $V_{1} \times V_{2} = W_{1}$. Similarly, we have $V_{2} \times V_{1} = W_{2}$. As $ (x,y) \in E_{1} \times E_{2} $ is an arbitrary element, 
  \begin{equation*}
      \Vert (x,y) - (T(x,y), T(y,x)) \Vert = r \mbox{ for all }  (x,y)\in E_{1}\times E_{2}
  \end{equation*}
  For each $(\tilde{y},\tilde{x}) \in E_2 \times E_1$, $T(\tilde{y},\tilde{x}) \in E_1$ and $T(\tilde{x},\tilde{y}) \in E_2$ we have 
  \begin{align*}
  r &=  \Vert (T(\tilde{y}, \tilde{x}), T(\tilde{x}, \tilde{y}))-(T(T(\tilde{y}, \tilde{x}), T(\tilde{x},\tilde{y}),T(T(\tilde{x}, \tilde{y}), T(\tilde{y},\tilde{x})))\Vert\\
  & \leq \frac{1}{2}\Big [ \Vert(\tilde{y},\tilde{x})-(T(\tilde{y}, \tilde{x}), T(\tilde{x}, \tilde{y}))\Vert+ \Vert ( \tilde{x},\tilde{y})-(T(\tilde{x}, \tilde{y}), T(\tilde{y}, \tilde{x}))\Vert \Big]\\
  & \leq  \frac{1}{2}[\Vert(\tilde{y},\tilde{x})-(T(\tilde{y}, \tilde{x}), T(\tilde{x}, \tilde{y}))\Vert+r].
  \end{align*}
Using (\ref{equation2.17}), it follows that $\Vert(\tilde{y},\tilde{x})-(T(\tilde{y}, \tilde{x}), T(\tilde{x}, \tilde{y}))\Vert= r$ for each $ (\tilde{y},\tilde{x}) \in E_{2} \times E_{1} $. This implies that $r = \Vert (T(x,y), T(y,x)) - (T(T(x,y), T(y,x)), T(T(y,x), T(x,y))) \Vert < \Vert (x,y) - (T(x,y), T(y,x)) \Vert = r$, a contradiction. 
\end{proof}
It is noted that the coupled best proximity point established in the Theorem \ref{theorem2.12} is not necessarily unique, as demonstrated by the following example:
\begin{example}
\emph{Consider the nonreflexive Banach space $l_{1}$ and $\{e_{n}\}$ be the canonical basis of $l_{1}$. Suppose $A = \overline{\mbox{con}}(\{ e_{2n-1}+ e_{2n} : n \in \mathbb{N} \})$, and  $B = \overline{\mbox{con}}(\{ e_{2n}+ e_{2n+1} : n \in \mathbb{N} \}) $. Then $ (A,B) $ is a bounded, closed and convex sets in $l_{1}$ so that $dist(A,B) = 2$. Indeed, for $ x\in A$, we have $$x = \sum\limits_{i=1}^k t_{i}( e_{{2n_i-1}} + e_{{2n_i}}), \mbox{ where } t_{i} \geq 0 \mbox{ and } \sum\limits_{i=1}^k t_{i} = 1$$ for $ y \in B$, we have $$y = \sum\limits_{j=1}^l s_{j}(\
    e_{{2n_j}} + e_{{2n_j+1}}) \in B \mbox{ where }  s_{j} \geq 0 \mbox{ and }  \sum\limits_{j=1}^l s_{j} = 1.$$
    Define $T:(A \times B \cup (B \times A) \rightarrow A \cup B$ by   
\[  T(x,y) =   \left \{ \begin{array}{rcl} e_{2}+e_{3} & \mbox{ if } (x,y) \in A \times B
         &  \\ e_{1}+e_{2} & \mbox{ if } (x,y) \in B\times A
         & 
    \end{array}\right.
    \]  
 Now, we prove that $T$ is a $p$-cyclic Kannan nonexpansive mapping. Clearly $T(A,B) \subseteq B $ and   $T(B,A) \subseteq A $. We need to show that   
     \begin{equation}\label{equation2.18}
      \Vert T(x_{1},y_{1})-T(x_{2},y_{2}) \Vert  \leq \frac{1}{2} \Big[ \Vert (x_{1},y_{1})-(T(x_{1},y_{1}), T(y_{1},x_{1}))+(x_{2},y_{2})-(T(x_{2},y_{2}), T(y_{2},x_{2}))\Big].
     \end{equation}
If $(x_{1},y_{1}) \in A \times B $ and $ (x_{2},y_{2}) \in A \times B$, then $T(x_{1},y_{1})=e_2+e_3= T(x_{2},y_{2})$ and (\ref{equation2.18}) holds trivially. Similarly, if $(x_{1},y_{1}) \in B \times A$ and $(x_{2},y_{2}) \in B \times A$, then also (\ref{equation2.18}) holds trivially.
If $(x_{1},y_{1}) \in A \times B$  and  $(x_{2},y_{2}) \in B \times A$, then $\Vert T(x_{1},y_{1})-T(x_{2},y_{2}) \Vert = \Vert e_{2}+e_{3}-(e_{1}+e_{2}) \Vert = 2$. Also, $\Vert x_{1}-T(x_{1},y_{1}) \Vert = \Vert x_{1}-(e_{2}+e_{3}) \Vert \leq 4$ and  $\Vert y_{1}-T(y_{1},x_{1}) = \Vert y_{1}-(e_{1}+e_{2}) \Vert \leq 4$. This implies that (\ref{equation2.18}) holds. For $(x,y) \in A \times B$,we have $\Vert (T(x,y)-T(T(x,y), T(y,x))\Vert = \Vert e_{2}+e_{3}-T(e_{2}+e_{3},e_{1}+e_{2}) \Vert = \Vert e_{2}+e_{3}-(e_{1}+e_{2}) \Vert = 2$ and 
$\Vert T(y,x)-T(T(y,x),T(x,y)) \Vert=\Vert e_1+e_2-T(e_1+e_2,e_2+e_3) = \Vert e_{1}+e_{2}-(e_{2}+e_{3}) \Vert = 2$. Also,  
$\Vert (x,y)-(T(x,y),T(y,x)) \Vert = \max \big \{ \Vert x-T(x,y) \Vert, \Vert y-T(y,x) \Vert \} \leq 4$. This gives
$\Vert (T(x,y), T(y,x))-(T(T(x,y), T(y,x)), T(T(y,x), T(x,y))) \Vert  < \Vert (x,y)-(T(x,y), T(y,x))-(T(T(x,y) \Vert$ with  $dist(A,B) < \Vert (x,y)-(T(x,y), T(y,x))-(T(T(x,y) \Vert$. Therefore, all the conditions of Theorem \ref{theorem2.12} are satisfied which yields the existence of a coupled best proximity point of $T$. We observe that 
\begin{align*}
 \Vert (e_{1}+e_{2}-T(e_{1}+e_{2},e_{2}+e_{3})\Vert = \Vert (e_{1}+e_{2}-(e_{2}+e_{3}) \Vert = 2 = dist(A,B),\\ 
 \Vert e_{2}+e_{3}-T(e_{2}+e_{3},e_{1}+e_{2}) \Vert = \Vert e_{2}+e_{3}-(e_{1}+e_{2}) \Vert = 2 = dist(A,B).   
\end{align*}
Therefore, $(e_1+e_2,e_2+e_3)$ is a coupled best proximity point of $T$. Also, $\Big\Vert \frac{1}{2}(e_{1}+e_{2})+\frac{1}{2}(e_{3}+e_{4})-T(\frac{1}{2}(e_{1}+e_{2})+\frac{1}{2}(e_{3}+e_{4}), (e_{2}+e_{3})) \Big\Vert = \Big\Vert \frac{1}{2}(e_{1}+e_{2})+\frac{1}{2}(e_{3}+e_{4})-(e_2+e_3)\Big\Vert =2 = dist(A,B)$ and $\Big\Vert e_{2}+e_{3}-T(e_{2}+e_{3}, \frac{1}{2}(e_{1}+e_{2})+\frac{1}{2}(e_{3}+e_{4})) \Big\Vert = \Vert e_{2}+e_{3}-(e_{1}+e_{2}) \Vert = 2 = dist(A,B)$. This implies that $(\frac{1}{2}(e_{1}+e_{2})+\frac{1}{2}(e_{3}+e_{4}), (e_{2}+e_{3}))$  is also a coupled best proximity points of $T$.}
\end{example} 

\section*{Acknowledgements}
The first author is grateful to the Government of India for providing financial support in the form of Institutional Fellowship.

\end{document}